\documentclass{amsart}
\usepackage[utf8]{inputenc}
\usepackage{amssymb}
\usepackage{hyperref}
\usepackage[final]{showkeys} 

\input xy
\xyoption{all}

\theoremstyle{definition}
\newtheorem{mydef}{Definition}[section]
\newtheorem{lem}[mydef]{Lemma}
\newtheorem{thm}[mydef]{Theorem}

\newtheorem{cor}[mydef]{Corollary}

\newtheorem{defin}[mydef]{Definition}
\newtheorem{example}[mydef]{Example}
\newtheorem{remark}[mydef]{Remark}

\newtheorem{fact}[mydef]{Fact}

\newcommand{\fct}[2]{{}^{#1}#2}

\newcommand{\ba}{\bar{a}}
\newcommand{\bb}{\bar{b}}

\newcommand{\bigM}{\widehat{M}}
\newcommand{\bigN}{\widehat{N}}

\newcommand{\bigK}{\widehat{\K}}
\newcommand{\sbigK}{\widehat{K}}

\newcommand{\bigtau}{\hat{\tau}}

\newcommand{\dom}[1]{\text{dom}(#1)}

\newcommand{\seq}[1]{\langle #1 \rangle}
\newcommand{\rest}{\upharpoonright}

\newcommand{\id}{\text{id}}

\newcommand{\leap}[1]{\le_{#1}}
\newcommand{\ltap}[1]{<_{#1}}

\newcommand{\lta}{\ltap{\K}}
\newcommand{\lea}{\leap{\K}}

\newcommand{\K}{\mathbf{K}}
\newcommand{\C}{\mathcal{C}}

\newbox\noforkbox \newdimen\forklinewidth
\forklinewidth=0.3pt \setbox0\hbox{$\textstyle\smile$}
\setbox1\hbox to \wd0{\hfil\vrule width \forklinewidth depth-2pt
 height 10pt \hfil}
\wd1=0 cm \setbox\noforkbox\hbox{\lower 2pt\box1\lower
2pt\box0\relax}
\def\unionstick{\mathop{\copy\noforkbox}\limits}

\def\1nf{\unionstick^{(1)}}

\def\2nf{\unionstick^{(2)}}
\def\3nf{\unionstick^{(3)}}

\newcommand{\gtp}{\text{gtp}}

\newcommand{\gS}{\operatorname{gS}}
\newcommand{\gSna}{\gS^{\text{na}}}

\newcommand{\Ll}{\mathbb{L}}

\newcommand{\Eat}{E_{\text{at}}}
\newcommand{\cl}{\operatorname{cl}}

\newcommand{\preim}{\operatorname{preim}}

\newcommand{\LS}{\text{LS}}

\title[Quasiminimal AECs]{Quasiminimal abstract elementary classes}
\date{\today }

\author{Sebastien Vasey}
\address{Department of Mathematical Sciences, Carnegie Mellon University, Pittsburgh, Pennsylvania, USA}
\email{sebv@cmu.edu}
\urladdr{http://math.cmu.edu/\textasciitilde svasey/}

\begin{document}

\begin{abstract}
  We propose the notion of a quasiminimal abstract elementary class (AEC). This is an AEC satisfying four semantic conditions: countable Löwenheim-Skolem-Tarski number, existence of a prime model, closure under intersections, and uniqueness of the generic orbital type over every countable model. We exhibit a correspondence between Zilber's quasiminimal pregeometry classes and quasiminimal AECs: any quasiminimal pregeometry class induces a quasiminimal AEC (this was known), and for any quasiminimal AEC there is a natural functorial expansion that induces a quasiminimal pregeometry class. We show in particular that the exchange axiom is redundant in Zilber's definition of a quasiminimal pregeometry class.
\end{abstract}

\maketitle
\tableofcontents

\section{Introduction}

Quasiminimal pregeometry classes were introduced by Zilber \cite{zil05} in order to prove a categoricity theorem for pseudo-exponential fields. Quasiminimal pregeometry classes are a class of structures carrying a pregeometry satisfying several axioms. Roughly (see Definition \ref{quasimin-class-def}) the axioms specify that the countable structures are quite homogeneous and that the generic type over them is unique (where types here are syntactic quantifier-free types). The original axioms included an ``excellence'' condition, but it has since been shown \cite{quasimin-five} that this follows from the rest. Zilber showed that a quasiminimal pregeometry class has at most one model in every uncountable cardinal, and in fact the structures are determined by their dimension. Note that quasiminimal pregeometry classes are typically non-elementary (see \cite[\S5]{quasimin}): they are axiomatizable in $\Ll_{\omega_1, \omega} (Q)$ (where $Q$ is the quantifier ``there exists uncountably many'') but not even in $\Ll_{\omega_1, \omega}$.

The framework of abstract elementary classes (AECs) was introduced by Saharon Shelah \cite{sh88} and encompasses for example classes of models of an $\Ll_{\omega_1, \omega} (Q)$ theory. Therefore quasiminimal pregeometry classes can be naturally seen as AECs (see Theorem \ref{qm-aec}). In this paper, we show that a converse holds: there is a natural class of AECs, which we call the \emph{quasiminimal AECs}, that corresponds to quasiminimal pregeometry classes. Quasiminimal AECs are required to satisfy four purely semantic properties (see Definition \ref{quasimin-def}), the most important of  which are that the AEC must, in a technical sense, be closed under intersections (this is called ``admitting intersections'', see Definition \ref{intersec-def}) and over each countable model $M$ there must be a \emph{unique} orbital (Galois) type that is not realized inside $M$.

It is straightforward (and implicit e.g.\ in \cite[\S4]{quasimin}, see also \cite[2.87]{group-config-kangas-apal}) to see that any quasiminimal pregeometry class is a quasiminimal AEC, but here we prove a converse (Theorem \ref{aec-qm}). We have to solve two difficulties:

\begin{enumerate}
\item  The axioms of quasiminimal pregeometry classes are very syntactic because they are phrased in terms of quantifier-free types. For example, one of the axioms (II(\ref{ii-2}) in Definition \ref{quasimin-class-def}) specifies that the models must have some syntactic homogeneity.
\item Nothing in the definition of quasiminimal AECs says that the models must carry a pregeometry. It is not clear that the natural closure $\cl^M (A)$ given by the intersections of all the $\K$-substructures of $M$ containing $A$ satisfies exchange.
\end{enumerate}

To get around the first difficulty, we use an argument from \cite[\S5]{quasimin-five} together with the technique of adding relation symbols for small Galois types to the vocabulary (called the Galois Morleyization in \cite{sv-infinitary-stability-afml}). To get around the second difficulty, we develop new tools to prove the exchange axiom of pregeometries in any setup where we know that the other axioms of pregeometries hold. We show (Corollary \ref{pregeom-cor}) that any \emph{homogeneous} closure space satisfying the finite character axiom of pregeometries also satisfies the exchange axiom (to the best of our knowledge, this is new\footnote{Although related to the study of quasiminimal structures in Itai-Tsuboi-Wakai \cite{itai-tsuboi-wakai} and later Pillay-Tanović \cite{pillay-tanovic}, Corollary \ref{pregeom-cor} is different. It gives a stronger conclusion using stronger hypotheses, see Remark \ref{previous-work-rmk}.}). As a consequence, the exchange axiom is redundant in the definition of a quasiminimal pregeometry class (Corollary \ref{exchange-not-necessary})\footnote{Interestingly, exchange was initially not part of Zilber's definition of quasiminimal pregeometry classes (see \cite[\S5]{zilber-pseudoexp}) but was added later. Some sources claim that the axiom is necessary, see \cite[Remark 2.24]{baldwinbook09} or \cite[p.~554]{quasimin}, but this seems to be due to a related counterexample that does not fit in the framework of quasiminimal pregeometry classes (see the discussion in Remark \ref{previous-work-rmk}).}.

An immediate corollary of the correspondence between quasiminimal AECs and quasiminimal pregeometry classes is that a quasiminimal AEC has at most one model in every uncountable cardinal (Corollary \ref{final-cor}). This can be seen as a generalization of the fact that algebraically closed fields of a fixed characteristic are uncountably categorical (indeed, algebraically closed fields are closed under intersections and if $F$ is a field, $a, b$ are transcendental over $F$, then $a$ and $b$ satisfy the same type over $F$).

Throughout this paper, we assume some basic familiarity with AECs (see \cite{baldwinbook09}), although we repeat the basic definitions. We use the notation from \cite{sv-infinitary-stability-afml}. In particular, we use $|M|$ to denote the universe of a structure $M$, and $\|M\|$ for its cardinality.

This paper was written while working on a Ph.D.\ thesis under the direction of Rami Grossberg at Carnegie Mellon University and I would like to thank Professor Grossberg for his guidance and assistance in my research in general and in this work specifically. I also thank John Baldwin, Will Boney, Levon Haykazyan, Jonathan Kirby, and Boris Zilber for helpful feedback on an early draft of this paper. Finally, I thank several anonymous referees for comments that helped improve the paper.

\section{Exchange in homogeneous closure spaces}\label{sec-2}

Pregeometries are a fundamental tool in geometric stability theory \cite{elements-geometric, pillay-geometric-stability}. They occur in the study of strongly minimal sets (where algebraic closure induces a pregeometry) and their generalization, regular types (where forking induces a pregeometry). In this section, we study closure spaces, which are objects satisfying the monotonicity and transitivity axioms of pregeometries. We want to know whether they satisfy the exchange axiom when they are homogeneous. We give criteria for when this is the case (Corollary \ref{pregeom-cor}). To the best of our knowledge, this is new (but see Remark \ref{previous-work-rmk}).

The following definition is standard, see e.g.\ \cite{crapo-rota-geometries}.

\begin{defin}
  A closure space is a pair $W = (X, \cl)$, where:

  \begin{enumerate}
  \item $X$ is a set.
  \item $\cl : \mathcal{P} (X) \rightarrow \mathcal{P} (X)$ satisfies:

    \begin{enumerate}
    \item Monotonicity: For any $A \subseteq X$, $A \subseteq \cl (A)$.
    \item Transitivity: For any $A, B \subseteq X$, $A \subseteq \cl (B)$ implies $\cl (A) \subseteq \cl (B)$.
    \end{enumerate}
  \end{enumerate}

  We write $|W|$ for $X$ and $\cl^W$ for $\cl$ (but when $W$ is clear from context we might forget it). For $a \in A$, we will often write $\cl (a)$ instead of $\cl (\{a\})$. Similarly, for sets $A, B \subseteq |W|$ and $a \in |W|$, we will write $\cl (Aa)$ instead of $\cl (A \cup \{a\})$ and $\cl (AB)$ instead of $\cl (A \cup B)$.
\end{defin}

\begin{defin}\label{basic-def}
  Let $W$ be a closure space.

  \begin{enumerate}
  \item For closure spaces $W_1, W_2$, we say that a function $f: |W_1| \rightarrow |W_2|$ is an \emph{isomorphism} if it is a bijection and for any $A \subseteq |W_1|$, $f[\cl^{W_1} (A)] = \cl^{W_2} (f[A])$. When $W_1 = W_2 = W$, we say that $f$ is an \emph{automorphism} of $W$.
  \item We say that $A \subseteq |W|$ is \emph{closed} if $\cl^W (A) = A$.
  \item\label{homog-def} For $\mu$ an infinite cardinal, we say that $W$ is \emph{$\mu$-homogeneous} if for any set $A$ with $|A| < \mu$ and any $a, b \in |W| \backslash \cl^{W} (A)$, there exists an automorphism of $W$ that fixes $A$ pointwise and sends $a$ to $b$.

  \item Let $\LS (W)$ be the least infinite cardinal $\mu$ such that for any $A \subseteq |W|$, $|\cl^W (A)| \le |A| + \mu$.
  \item Let $\kappa (W)$ be the least infinite cardinal $\kappa$ such that for any $A \subseteq |W|$, $a \in \cl^W (A)$ implies that there exists $A_0 \subseteq A$ with $|A_0| < \kappa$ and $a \in \cl^W (A_0)$. We say that $W$ has \emph{finite character} if $\kappa (W) = \aleph_0$.
  \item We say that $W$ has \emph{exchange over $A$} if $A \subseteq |W|$ and for any $a, b$, if $a \in \cl^W (Ab) \backslash \cl^W (A)$, then $b \in \cl^W (Aa)$. We say that $W$ has \emph{exchange} if it has exchange over every $A \subseteq |W|$.
  \item We say that $W$ is a \emph{pregeometry} if it has finite character and exchange.
  \end{enumerate}
\end{defin}

\begin{remark} \
  \begin{enumerate}
  \item $\LS (W) \le \|W\| + \aleph_0$ and $\kappa (W) \le \|W\|^+ + \aleph_0$.
  \item $\LS (W) \le \kappa (W) \cdot \sup_{A \subseteq |W|, |A| < \kappa (W)} |\cl^W (A)|$.
  \end{enumerate}
\end{remark}

\begin{defin}
  For $A \subseteq |W|$, let $W_A$ be the following closure space: $|W_A| := |W| \backslash A$, and $\cl^{W_A} (B) := \cl^W (AB) \cap |W_A|$.
\end{defin}

\begin{lem}\label{exchange-lem}
  Let $W$ be a closure space.
  \begin{enumerate}
  \item For $\mu$ an infinite cardinal, if $W$ is $\mu$-homogeneous, $A \subseteq |W|$ and $|A| < \mu$, then $W_A$ is $\mu$-homogeneous.
  \item $W$ has exchange over $A$ if and only if $W_A$ has exchange over $\emptyset$.
  \item $W$ has exchange if and only if $W$ has exchange over every $A$ with $|A| < \kappa (W)$.
  \end{enumerate}
\end{lem}
\begin{proof}
  Straightforward.
\end{proof}

Closure spaces where exchange always fails are studied in the literature under the names ``antimatroid'' or ``convex geometry'' \cite{theory-convex-geometries}. One of the first observations one can make is that there is a natural ordering in this context:

\begin{defin}\label{i-def}
  Let $W$ be a closure space. For $a, b \in |W|$, say $a \le b$ if $a \in \cl (b)$. We say $a < b$ if $a \le b$ but $b \not \le a$. We write $a \sim b$ if both $a \le b$ and $b \le a$. We denote by $I(W)$ the partial order on $|W| / \sim$ induced by $\le$.
\end{defin}
\begin{remark}
By the transitivity axiom, $(|W|, \le)$ is indeed a pre-order. Moreover any automorphism of $W$ induces an automorphism of $(|W|, \le)$ and hence of $I (W)$.
\end{remark}
\begin{remark}
  Let $W$ be a closure space where $\emptyset$ is closed. Then $W$ fails exchange over $\emptyset$ if and only if there exists $a, b \in |W|$ such that $a < b$.
\end{remark}

To give conditions under which exchange follows from homogeneity, we will study the ordering $I(W)$ from Definition \ref{i-def}. The key is:

\begin{lem}\label{dlo-lem}
  If $W$ is $\aleph_0$-homogeneous and $\emptyset$ is closed, then $I(W)$ is both 1-transitive and 2-transitive. More precisely, for any $a$ and $c$ in $W$, there is an automorphism of $W$ sending $a$ to $c$ and if $b \not \le a$ and $d \not \le c$, then there exists an automorphism of $W$ sending $(a, b)$ to $(c, d)$. In particular if $W$ fails exchange over $\emptyset$, then $I (W)$ is a dense linear order without endpoints.
\end{lem}
\begin{proof}
  We have that $d \notin \cl (c)$ and $b \notin \cl (a)$. By $\aleph_0$-homogeneity, there exists an automorphism $f$ of $W$ taking $c$ to $a$ (using that $\emptyset$ is closed, so $a, c \notin \cl (\emptyset) = \emptyset$. Let $d' := f (d)$. Then $d' \notin \cl (a)$. By $\aleph_0$-homogeneity, there exists an automorphism $g$ of $W$ fixing $a$ and sending $d'$ to $b$. Let $h := f^{-1} \circ g^{-1}$. Then $h (a) = c$ and $h (b) = d$, as desired. Now if exchange over $\emptyset$ fails, there exists $a, b$ such that $a < b$. By 2-transitivity, any $c, d$ with $c \not \le d$ must satisfy $d < c$. Thus $I(W)$ is linear. We similarly obtain (using 1-transitivity) that $I(W)$ is dense and without endpoints.
\end{proof}

\begin{thm}\label{pregeom-emptyset-thm}
  Let $W$ be an $\aleph_0$-homogeneous closure space where $\emptyset$ is closed. Then $W$ has exchange over $\emptyset$ if at least one of the following conditions hold:

  \begin{enumerate}
  \item $\|W\| < \aleph_0$.
  \item $\|W\| \ge \LS (W)^{++}$.
  \item\label{emptyset-thm-3} $W$ is $\LS (W)^+$-homogeneous and $\kappa (W) = \aleph_0$.
  \end{enumerate}
\end{thm}
\begin{proof}
  Suppose for a contradiction that exchange over $\emptyset$ fails.
  
  For $b \in |W|$, write $(-\infty, b) := \{a \in |W| : a < b\}$, and similarly for $(-\infty, b]$. Note that if $A \subseteq |W|$ is closed and $a \in A$, then by definition of $\le$ and the transitivity axiom, $(-\infty, a] \subseteq A$. Similarly, if $b \notin A$ then by Lemma \ref{dlo-lem} $A \subseteq (-\infty, b)$.

      \begin{enumerate}
      \item If $\|W\| < \aleph_0$, then Lemma \ref{dlo-lem} directly gives a contradiction.
  \item Let $A \subseteq |W|$ be closed such that $|A| \le \LS (W)$ and let $B \subseteq |W|$ be closed with $A \subseteq B$ and $|B| = \LS (W)^+$. Let $a \in A$ and let $b \notin B$. Then $(-\infty, a) \subseteq A$ and $B \subseteq (-\infty, b)$. Therefore $|(-\infty, a)| \le \LS (W)$ and $|(-\infty, b)| \ge \LS (W)^+$. However by homogeneity there exists an automorphism of $W$ sending $a$ to $b$, a contradiction.
  \item We first prove two claims.
    
    \underline{Claim 1}: If $b \in |W|$, then $\cl ((-\infty, b)) = (-\infty, b]$.
      
      \underline{Proof of Claim 1}: Let $B := \cl(-\infty, b)$. First note that $B \subseteq \cl (b)$, hence $|B| \le \LS (W)$ (so we can apply homogeneity to it) and $B \subseteq (-\infty, b]$. By monotonicity, $(-\infty, b) \subseteq B$. Also, if $B \neq (-\infty, b)$, then $b \in B$ (say $c \in B \backslash (-\infty, b)$. Then $c \not < b$, so by Lemma \ref{dlo-lem}, $b \le c$, so since $B$ is closed $b \in B$). Thus if $b \notin B$, then $B = (-\infty, b)$. This is impossible: take $c \in |W|$ such that $b < c$ (exists by Lemma \ref{dlo-lem}). Then there is an automorphism of $W$ taking $b$ to $c$ fixing $B$, which is impossible as $b$ is a least upper bound of $B$ but $c$ is not. Therefore $(-\infty, b] \subseteq B$. $\dagger_{\text{Claim 1}}$

          \underline{Claim 2}: If $\seq{A_i : i \in I}$ is a non-empty collection of subsets of $|W|$, then $\cl (\bigcup_{i \in I} A_i) = \bigcup_{i \in I} \cl (A_i)$.
          
          \underline{Proof of Claim 2}: Clearly, the right hand side is contained in the left hand side. We show the other inclusion. Let $A := \bigcup_{i \in I} A_i$. Let $a \in \cl (A)$. By finite character, there exists a finite $A' \subseteq A$ such that $a \in \cl (A')$. Since $\emptyset$ is closed, $A'$ cannot be empty. Say $A' = \{a_0, \ldots, a_{n - 1}\}$, with $a_0 \le a_1 \le \ldots \le a_{n - 1}$ (we are implicitly using Lemma \ref{dlo-lem}). Then $a \in \cl (a_{n - 1})$. Pick $i \in I$ such that $a_{n - 1} \in A_i$. Then $a \in \cl (A_i)$, as desired. $\dagger_{\text{Claim 2}}$

          Now pick any $b \in |W|$. Note that (using Lemma \ref{dlo-lem}) $(-\infty, b) = \bigcup_{a < b} (-\infty, a) = \bigcup_{a < b} (-\infty, a]$. However on the one hand, by Claim 1, $\cl (\bigcup_{a < b} (-\infty, a)) = \cl ((-\infty, b)) = (-\infty, b]$ but on the other hand, by Claim 2, $\cl (\bigcup_{a < b} (-\infty, a)) = \bigcup_{a < b} \cl ((-\infty, a)) = \bigcup_{a < b} (-\infty, a] = (-\infty, b)$, a contradiction.
  \end{enumerate}
\end{proof}

\begin{cor}\label{pregeom-cor}
  Let $W$ be a $(\kappa (W) + \LS (W)^+)$-homogeneous closure space. If either $\kappa (W) = \aleph_0$ or $\|W\| \notin [\aleph_0, \LS (W)^{+}]$, then $W$ has exchange.
\end{cor}
\begin{proof}
  Let $\mu := \kappa (W) + \LS (W)^+$. By Lemma \ref{exchange-lem}, it is enough to see that $W$ has exchange over every set $A$ with $|A| < \kappa (W)$. Fix such an $A$. By Lemma \ref{exchange-lem}, it is enough to see that $W' := W_{\cl^W (A)}$ has exchange over $\emptyset$. Note that $W'$ is $\mu$-homogeneous and $\LS (W') \le \LS (W)$. Moreover $\emptyset$ is closed in $W'$. Observe that $\kappa (W) = \aleph_0$ implies that $\kappa (W') = \aleph_0$, $\|W'\| \le \|W\|$, but $\|W\| \ge \LS (W)^{++}$ implies that $\|W'\| \ge \LS (W)^{++}$. Therefore by Theorem \ref{pregeom-emptyset-thm}, $W'$ has exchange over $\emptyset$, as desired.
\end{proof}

We give a few examples showing that the hypotheses of Corollary \ref{pregeom-cor} are near optimal:

\begin{example}\label{example-pregeom} \
  \begin{enumerate}
  \item On any partial order $\mathbb{P}$, one can define a closure operator $\cl_1$ by $\cl_1 (A) := \{b \in \mathbb{P} \mid \exists a \in A : b \le a\}$. The resulting closure space $W_{1, \mathbb{P}}$ has exchange over $\emptyset$ if and only if there are no $a, b \in \mathbb{P}$ with $a < b$. Note that if $\mathbb{P}$ is e.g.\ a dense linear order, then $W_{1, \mathbb{P}}$ is not $\aleph_1$-homogeneous.
  \item On the other hand, one can define $\cl_2 (A) := \cl_1 (A) \cup \{b \in \mathbb{P} \mid \forall c (c < b \rightarrow c \in \cl_1 (A))\}$. This gives a closure space $W_{2, \mathbb{P}}$. Setting $\mathbb{P} := \mathbb{Q}$, $W_{2, \mathbb{Q}}$ is $\aleph_1$-homogeneous and does not have exchange over $\emptyset$ but note that $\kappa (W_{2, \mathbb{Q}}) = \aleph_1$, because the statement ``$0 \in \cl ((-\infty, 0))$'' is not witnessed by a finite subset of $(-\infty, 0)$.
  \item The closure space $W_{2, \mathbb{Q} \times \omega_1}$ (where $\mathbb{Q} \times \omega_1$ is ordered by the reverse lexicographical ordering, i.e.\ the second component is the most significant) is also $\aleph_1$-homogeneous, satisfies $\LS (W_{2, \mathbb{Q} \times \omega_1}) = \aleph_0$, $\kappa (W_{2, \mathbb{Q} \times \omega_1}) = \aleph_1$, and does not have exchange over $\emptyset$.
  \end{enumerate}
\end{example}

\begin{remark}\label{previous-work-rmk}
  In \cite[\S5]{pillay-tanovic}, Pillay and Tanović generalize an earlier result of Itai, Tsuboi, and Wakai \cite[2.8]{itai-tsuboi-wakai}) by proving (roughly) that any quasiminimal structure of size at least $\aleph_2$ induces a pregeometry. Here, we call a structure \emph{quasiminimal} if every definable set is either countable or co-countable. Thus the Pillay-Tanović result is a (more general) version of Corollary \ref{pregeom-cor} for the case $\kappa (W) = \aleph_0$, $\|W\| \ge \aleph_2$, and $\LS (W) = \aleph_0$.

  Note that in the Pillay-Tanović context the hypothesis that the size should be at least $\aleph_2$ is needed: consider \cite[Example 2.2(3a)]{itai-tsuboi-wakai} the structure $M := (\mathbb{Q} \times \omega_1, <)$ (where as above $<$ denotes the reverse lexicographical ordering). The closure space induced by $M$ is the same as $W_{1, \mathbb{Q} \times \omega_1}$ from Example \ref{example-pregeom}, so it does not have exchange. Note that $M$ is homogeneous in the model-theoretic sense that every countable partial elementary mapping from $M$ into $M$ can be extended (and also in the syntactic sense of \cite[\S4]{pillay-tanovic}), but this does \emph{not} make the corresponding closure space homogeneous in the sense of Definition \ref{basic-def}(\ref{homog-def}). Indeed, two elements could satisfy the same first-order type but not the same type e.g.\ in an infinitary logic. This is used in the proof of Theorem \ref{pregeom-emptyset-thm}(\ref{emptyset-thm-3}): if $(I, <)$ is a dense linear order and $b < c$, then $b$ and $c$ will satisfy the same first-order type over $(-\infty, b)$, but there cannot be an automorphism sending $b$ to $c$ fixing $(-\infty, b)$. Thus $M$ cannot be a counterexample to Theorem \ref{pregeom-emptyset-thm}(\ref{emptyset-thm-3}). In the proof of Theorem \ref{pregeom-thm}, we will build a (Galois) saturated model $N$ and work with the pregeometry generated by a certain closure operator inside it. The (orbital) homogeneity of $N$ will give homogeneity of the pregeometry in the strong sense given here.
\end{remark}

\section{On AECs admitting intersections}

In this section, we review the definition of an AEC admitting intersections, first introduced by Baldwin and Shelah \cite[1.2]{non-locality}. We give a few known facts and show (Theorem \ref{intersec-up}) that admitting intersections transfers up in AECs: if all models of a fixed size above the Löwenheim-Skolem number admit intersections, then the entire class admits intersections.

We first recall the definition of an abstract elementary class, due to Shelah \cite{sh88}.

\begin{defin}\label{aec-def}
    An \emph{abstract elementary class} (AEC for short) is a pair $\K = (K, \lea)$, where:

  \begin{enumerate}
    \item\label{aec-1} $K$ is a class of $\tau$-structures, for some fixed vocabulary $\tau = \tau (\K)$. 
    \item $\lea$ is a partial order (that is, a reflexive and transitive relation) on $K$. 
    \item $(K, \lea)$ respects isomorphisms: If $M \lea N$ are in $K$ and $f: N \cong N'$, then $f[M] \lea N'$. In particular (taking $M = N$), $K$ is closed under isomorphisms.
    \item\label{aec-4} If $M \lea N$, then $M \subseteq N$. 
    \item\label{aec-5} Coherence: If $M_0, M_1, M_2 \in K$ satisfy $M_0 \lea M_2$, $M_1 \lea M_2$, and $M_0 \subseteq M_1$, then $M_0 \lea M_1$;
    \item Tarski-Vaught axioms: Suppose $\seq{M_i : i \in I}$ is a $\lea$-directed system in $\K$. Then:

        \begin{enumerate}

            \item $M := \bigcup_{i \in I} M_i \in K$ and $M_i \lea M$ for all $i \in I$.
            \item\label{smoothness-axiom} If there is some $N \in K$ so that for all $i \in I$ we have $M_i \lea N$, then we also have $M \lea N$.

        \end{enumerate}

      \item Löwenheim-Skolem-Tarski axiom: There exists a cardinal $\lambda \ge |\tau(\K)| + \aleph_0$ such that for any $M \in K$ and $A \subseteq |M|$, there is some $M_0 \lea M$ such that $A \subseteq |M_0|$ and $\|M_0\| \le |A| + \lambda$. We write $\LS (\K)$ for the minimal such cardinal.
  \end{enumerate}
\end{defin}

We will use $\K_\lambda$ to denote the restriction of $\K$ to models of size $\lambda$, and $\K_{\ge \lambda}$ to similarly denote the restriction of $\K$ to models of size at least $\lambda$. As in \cite{grossbergbook}, we call an \emph{abstract class} a pair $\K = (K, \lea)$ satisfying conditions (\ref{aec-1}) to (\ref{aec-4}) in Definition \ref{aec-def}. We say that an abstract class is \emph{coherent} if it also satisfies (\ref{aec-5}).

Note that any AEC is a coherent abstract class, and if $\K$ is an AEC and $\lambda$ is a cardinal, then $\K_\lambda$ is also a coherent abstract class.

\begin{defin}\label{intersec-def}
  Let $\K$ be a coherent abstract class. Let $N \in \K$ and let $A \subseteq |N|$.

  \begin{enumerate}
  \item Let $\cl^N (A)$ be the set $\bigcap \{|M| : M \lea N \land A \subseteq |M|\}$. Note that $\cl^N (A)$ induces a $\tau (\K)$-substructure of $N$, so we will abuse notation and also write $\cl^N (A)$ for this substructure.
  \item We say that $N$ \emph{admits intersections over $A$} if $\cl^N (A) \lea N$ (more formally, there exists $M \lea N$ such that $|M| = \cl^N (A)$).
  \item We say that $N$ \emph{admits intersections} if it admits intersections over all $A \subseteq |N|$.
  \item We say that $\K$ \emph{admits intersections} if every $N \in \K$ admits intersections.
  \end{enumerate}
\end{defin}

\begin{lem}\label{intersec-closure-lem}
  Let $\K$ be a coherent abstract class and let $N \in \K$. Then $(|N|, \cl^N)$ is a closure space and any $M \lea N$ is closed.
\end{lem}
\begin{proof}
  Immediate from the definition of $\cl^N$.
\end{proof}

The following characterization of admitting intersections in terms of the existence of a certain closure operator will be used often in this paper. The proof is similar to that of \cite[2.11]{ap-universal-apal}.

\begin{fact}\label{intersec-charact}
  Let $\K$ be a coherent abstract class and let $N \in \K$. The following are equivalent:

  \begin{enumerate}
  \item\label{intersec-charact-1} $N$ admits intersections.
  \item\label{intersec-charact-2} For every non-empty collection $S$ of $\K$-substructures of $N$, we have that $\bigcap S \lea N$. 
 \item\label{intersec-charact-3} There is a closure space $W$ such that:

    \begin{enumerate}
    \item $|W| = |N|$.
    \item The closed sets in $W$ are exactly the sets of the form $|M|$ for $M \lea N$.
    \end{enumerate}
  \end{enumerate}
\end{fact}

The next result is observed (for AECs) in \cite[2.14(4)]{ap-universal-apal}. The proof generalizes to coherent abstract classes.

\begin{fact}\label{cl-restr}
  Let $\K$ be a coherent abstract class and let $M \lea N$ both be in $\K$. Let $A \subseteq |M|$ If $N$ admits intersections over $A$, then $M$ admits intersections over $A$ and $\cl^M (A) = \cl^N (A)$.
\end{fact}

We now explain why admitting intersections transfers up. This is routine, so we only sketch the proof and leave the details to the reader.

\begin{thm}\label{intersec-up}
  Let $\K$ be an AEC and let $\lambda \ge \LS (\K)$. Let $N \in \K_{\ge \lambda}$. If $M$ admits intersections for all $M \in \K_\lambda$ with $M \lea N$, then $N$ admits intersections. In particular if $\K_\lambda$ admits intersections, then $\K_{\ge \lambda}$ admits intersections.
\end{thm}
\begin{proof}[Proof sketch]
  Let $A \subseteq |N|$. We want to show that $\cl^N (A) \lea N$. To see this, first prove that:

  $$
  \cl^N (A) = \bigcup_{M \lea N, M \in \K_{\lambda}} \cl^M (A \cap |M|)
  $$

  and then use the Tarski-Vaught chain axioms together with the fact that each $\cl^M (A \cap |M|)$ in the right hand side satisfies $\cl^M (A \cap |M|) \lea M \lea N$, since by hypothesis $M$ admits intersections.
\end{proof}

We will use two facts about AECs admitting intersections in the next section. First, the closure operator has finite character \cite[2.14(6)]{ap-universal-apal}:

\begin{fact}\label{cl-kappa}
  Let $\K$ be an AEC and let $N \in \K$. If $N$ admits intersections, then $\kappa ((|N|, \cl^N)) = \aleph_0$.
\end{fact}

To state the second fact, we first recall the definition of a Galois (or orbital) type. The definition is due Shelah (see for example \cite[II.1.9]{shelahaecbook}), but we use the notation from the preliminaries of \cite{sv-infinitary-stability-afml}.

\begin{defin}\label{gtp-def}
  Let $\K$ be an abstract class.
  
  \begin{enumerate}
    \item Let $\K^3$ be the set of triples of the form $(\bb, A, N)$, where $N \in \K$, $A \subseteq |N|$, and $\bb$ is a sequence of elements from $N$. 
    \item For $(\bb_1, A_1, N_1), (\bb_2, A_2, N_2) \in \K^3$, we say $(\bb_1, A_1, N_1)E_{\text{at}} (\bb_2, A_2, N_2)$ if $A := A_1 = A_2$, and there exists $f_\ell : N_\ell \xrightarrow[A]{} N$ such that $f_1 (\bb_1) = f_2 (\bb_2)$.
    \item Note that $E_{\text{at}}$ is a symmetric and reflexive relation on $\K^3$. We let $E$ be the transitive closure of $E_{\text{at}}$.
    \item For $(\bb, A, N) \in \K^3$, let $\gtp (\bb / A; N) := [(\bb, A, N)]_E$. We call such an equivalence class a \emph{Galois type}. 
    \item For $p = \gtp (\bb / A; N)$ a Galois type, define\footnote{It is easy to check that this does not depend on the choice of representatives.} $\ell (p) := \ell (\bb)$ and $\dom{p} := A$. For $M \in \K$, we also define:

      $$\gS (M) := \{\gtp (b / M; N) \mid M \lea N, b \in |N|\}$$

      and

      $$\gS^{<\omega} (M) := \{\gtp (\bb / M; N) \mid M \lea N, \bb \in \fct{<\omega}{|N|}\}$$

      In words, $\gS (M)$ is the set of all types \emph{of length one} over $M$ and $\gS^{<\omega} (M)$ is the set of all types of \emph{any finite length} over $M$.
  \end{enumerate}
\end{defin}

It is easy to check that $\Eat$ is transitive in any abstract class with amalgamation. AECs admitting intersections may not have amalgamation but Galois types are still nicely characterized there (see \cite[1.3(1)]{non-locality} or \cite[2.18]{ap-universal-apal}):

\begin{fact}\label{gtp-eq}
  Let $\K$ be a coherent abstract class admitting intersections. Then $\gtp (\bb_1 / A; N_1) = \gtp (\bb_2 / A; N_2)$ if and only if there exists $f: \cl^{N_1} (A \bb_1) \cong_A \cl^{N_2} (A \bb_2)$ such that $f (\bb_1) = \bb_2$.
\end{fact}

\section{Quasiminimal AECs}

In this section, we define quasiminimal AECs and show that they are essentially the same as quasiminimal pregeometry classes.

Following Shelah \cite[II.1.9(1A)]{shelahaecbook}, we will write $\gSna (M)$ for the set of \emph{nonalgebraic} types over $M$: that is, the set of $p \in \gS (M)$ such that $p = \gtp (a / M; N)$ with $a \notin |M|$ (in the context of this paper, there will be a unique nonalgebraic type which we will call \emph{the generic type}). We say that $M \in \K$ is \emph{prime} if for any $N \in \K$, there exists $f: M \rightarrow N$.

\begin{defin}\label{quasimin-def}
  An AEC $\K$ is \emph{quasiminimal} if:

  \begin{enumerate}
  \item\label{qm-1} $\LS (\K) = \aleph_0$.
  \item\label{qm-2} There is a prime model in $\K$.
  \item\label{qm-3} $\K_{\le \aleph_0}$ admits intersections.
  \item\label{qm-4} (Uniqueness of the generic type) For any $M \in \K_{\le \aleph_0}$, $|\gSna (M)| \le 1$.
  \end{enumerate}

  We say that $\K$ is \emph{unbounded} if it satisfies in addition:

\begin{enumerate}
  \setcounter{enumi}{4}
  \item\label{qm-5} There exists $\seq{M_i : i < \omega}$ strictly increasing in $\K$.
  \end{enumerate}
\end{defin}

For the convenience of the reader, we repeat here the definition of a quasiminimal pregeometry class. We use the numbering and presentation from Kirby \cite{quasimin}, see there for more details on the terminology. We omit axiom III (excellence), since it has been shown \cite{quasimin-five} that it follows from the rest. We have added axiom 0(\ref{0-3}) that also appears in Haykazyan \cite[2.2]{hayk2013-jsl} and corresponds to (\ref{qm-2}) in the definition of a quasiminimal AEC, as well as axiom 0(\ref{0-1}) which requires that the class be non-empty and that the vocabulary be countable (this can be assumed without loss of generality, see \cite[5.2]{quasimin}).

As in Definition \ref{quasimin-def}, we call the class \emph{unbounded} if it has an infinite dimensional model (this is the nontrivial case that interests us here).

\begin{defin}\label{quasimin-class-def}
  A \emph{quasiminimal pregeometry class} is a class $\C$ of pairs $(H, \cl_H)$, where $H$ is a $\tau$-structure (for a fixed vocabulary $\tau = \tau (\C)$) and $\cl_H : \mathcal{P} (|H|) \rightarrow \mathcal{P} (|H|)$ is a function, satisfying the following axioms:

  \begin{itemize}
  \item[0: ]
    \begin{enumerate}
    \item\label{0-1} $|\tau (\C)| \le \aleph_0$ and $\C \neq \emptyset$.
    \item\label{0-2} If $(H, \cl_H), (H', \cl_{H'})$ are both $\tau$-structures with functions on their powersets and $f: H \cong H'$ is also an isomorphism from $(|H|, \cl_H)$ onto $(|H'|, \cl_{H'})$, then $(H', \cl_{H'}) \in \C$.
    \item\label{0-3} If $(H, \cl_H), (H', \cl_{H'}) \in \C$, then $H$ and $H'$ satisfy the same quantifier-free sentences.
    \end{enumerate}
  \item[I: ]
    \begin{enumerate}
    \item\label{i-1} For each $(H, \cl_H) \in \C$, $(|H|, \cl_H)$ is a pregeometry such that the closure of any finite set is countable.
    \item\label{i-2} If $(H, \cl_H) \in \C$ and $X \subseteq |H|$, then the $\tau (\C)$-structure induced by $\cl_H (X)$ together with the appropriate restriction of $\cl_H$ is in $\C$.
    \item\label{i-3} If $(H, \cl_H), (H', \cl_{H'}) \in \C$, $X \subseteq |H|$, $y \in \cl_H (X)$, and $f: H \rightharpoonup H'$ is a partial embedding with $X \cup \{y\} \subseteq \preim (f)$, then $f (y) \in \cl_{H'} (f[X])$.
    \end{enumerate}
  \item[II: ]
    Let $(H, \cl_{H}), (H', \cl_{H'}) \in \C$. Let $G \subseteq H$ and $G' \subseteq H'$ be countable closed subsets or empty and let $g: G \rightarrow G'$ be an isomorphism.
    \begin{enumerate}
    \item\label{ii-1} If $x \in |H|$ and $x' \in |H'|$ are independent from $G$ and $G'$ respectively, then $g \cup \{(x, x')\}$ is a partial embedding.
    \item\label{ii-2} If $g \cup f : H \rightharpoonup H'$ is a partial embedding, $f$ has finite preimage $X$, and $y \in \cl_H (X \cup G)$, then there is $y' \in H'$ such that $g \cup f \cup \{(y, y')\}$ is a partial embedding.
    \end{enumerate}
  \item[IV: ]
    \begin{enumerate}
    \item $\C$ is closed under unions of increasing chains: If $\delta$ is a limit ordinal and $\seq{(H_i, \cl_{H_i}) : i < \delta}$ is increasing with respect to being a closed substructure (i.e.\ for each $i < \delta$, $H_i \subseteq H_{i + 1}$ and $\cl_{H_{i + 1}} \rest \mathcal{P} (|H_i|) = \cl_{H_i}$), then $(H_\delta, \cl_{H_\delta}) \in \C$, where $H_\delta = \bigcup_{i < \delta} H_i$ and $\cl_{H_\delta} (X) = \bigcup_{i < \delta} \cl_{H_i} (X \cap |H_i|)$.
    \end{enumerate}
  \end{itemize}

  We say that $\C$ is \emph{unbounded} if it satisfies in addition:
  \begin{itemize}
  \item[IV: ]
    \begin{enumerate}
      \setcounter{enumi}{1}
    \item\label{iv-2} $\C$ contains an infinite dimensional model (i.e.\ there exists $(H, \cl_H) \in \C$ with $\seq{a_i : i < \omega}$ in $H$ such that $a_i \notin \cl_H (\{a_j : j < i\})$ for all $i < \omega$).
    \end{enumerate}
  \end{itemize}
\end{defin}

It is straightforward to show that quasiminimal pregeometry classes are (after forgetting the pregeometry and ordering them with ``being a closed substructure'') quasiminimal AECs. That they are AECs is noted in \cite[\S4]{quasimin}. In fact, the exchange axiom is not necessary for this. The main point is that axiom II of quasiminimal pregeometry classes allows us to do a back and forth argument to prove the desired existence of prime models and the equality of any two nonalgebraic Galois types over a common model. We sketch a proof here for completeness.

\begin{defin}\label{k-of-c}
  Let $\C$ be a quasiminimal pregeometry class.

  \begin{enumerate}
  \item For $(H, \cl_H), (H', \cl_{H'}) \in \C$, we write $(H, \cl_H) \leap{\C} (H', \cl_{H'})$ if $H \subseteq H'$ and $\cl_{H'} \rest \mathcal{P} (|H|) = \cl_H$.
  \item Let $\K = \K (\C) := (K (\C), \lea)$ be defined as follows:
  \begin{enumerate}
  \item\label{k-of-c-1} $K (\C) := \{M \mid \exists \cl : (M, \cl) \in \C\}$.
  \item\label{k-of-c-2} $M \lea N$ if for some $\cl_M, \cl_N$, $(M, \cl_M), (N, \cl_N) \in \C$ and $ (M, \cl_M) \leap{\C} (N, \cl_N)$.
  \end{enumerate}
  \end{enumerate}
\end{defin}

Note that the structure determines the pregeometry (see also the discussion after \cite[1.2]{quasimin}), hence we do not loose any information by going from $\C$ to $\K (C)$:

\begin{lem}\label{uq-cl}
  Let $\C$ satisfy axioms 0 and I from the definition of a quasiminimal pregeometry class except that in I(\ref{i-1}) $\cl_H$ may not have exchange. Let $\K := \K (\C)$. Then $\K$ is a coherent abstract class admitting intersections and for any $M \in \K$, there exists a unique operator $\cl_M$ such that $(M, \cl_M) \in \C$. Moreover $\cl_M = \cl^M$ (see Definition \ref{intersec-def}).
\end{lem}
\begin{proof}
  If $M \in \K$, then by definition there exists $\cl_M$ such that $(M, \cl_M) \in \C$. Moreover if $(M, \cl) \in \C$ then by axiom I(\ref{i-3}) used with the identity embedding, $\cl = \cl_M$. It immediately follows from axiom I(\ref{i-2}) that $\K$ is a coherent abstract class. We then get, also using I(\ref{i-2}), that $\cl_M = \cl^M$, and thus using Fact \ref{intersec-charact} that $\K$ admits intersection.
\end{proof}
\begin{thm}\label{qm-aec}
  If $\C$ satisfies all the axioms of a quasiminimal pregeometry class except that in I(\ref{i-1}) $\cl_H$ may not have exchange, then $\K (\C)$ is a quasiminimal AEC. Moreover $\C$ is unbounded if and only if $\K (\C)$ is unbounded.
\end{thm}
\begin{proof}
  We will sometimes use Lemma \ref{uq-cl} without explicit mention. Let $\K := \K (\C)$. If $\C$ is bounded, there are no infinite increasing chains in $\K$ so by axioms 0 and I, $\K$ is an AEC. If $\C$ is unbounded, axioms 0, I, and IV similarly give that $\K$ is an AEC. Since the closure of any finite set is countable (axiom I(\ref{i-1})) and $|\tau (\C)| \le \aleph_0$ (axiom 0(\ref{0-1})), $\LS (\K) = \aleph_0$. This proves that (\ref{qm-1}) in Definition \ref{quasimin-def} holds.

  As for axiom (\ref{qm-2}), by axiom 0(\ref{0-1}), $\C \neq \emptyset$, hence $\K \neq \emptyset$. Let $M \in \K$ and let $M_0 := \cl^M (\emptyset)$. By Lemma \ref{uq-cl}, $\K$ admits intersections hence $M_0 \in \K$. We show that $M_0$ is the desired prime model. Let $N \in \K$. By axiom 0(\ref{0-3}), the empty map is a partial embedding from $M$ into $N$. Using axiom II(\ref{ii-2}) to do a back and forth argument (see the proof of \cite[2.1]{quasimin}), we can extend it to a map $f_0 : M_0 = \cl^M (\emptyset) \cong \cl^N (\emptyset)$. Since $\cl^N (\emptyset) \lea N$ (for the same reason that $M_0 \lea M$), $f_0$ witnesses that $M_0$ embeds into $N$, as desired.

  Lemma \ref{uq-cl} already showed that $\K_{\le \aleph_0}$ admit intersections. Let us check axiom (\ref{qm-4}) in Definition \ref{quasimin-def}. Let $M \in \K_{\le \aleph_0}$. We want to show that $|\gSna (M)| \le 1$. Let $p_1, p_2 \in \gSna (M)$. Say $p_\ell = \gtp (a_\ell / M; N_\ell)$, $\ell = 1,2$. We want to see that $p_1 = p_2$. Without loss of generality (since $\K_{\le \aleph_0}$ admits intersections), $N_\ell = \cl^{N_\ell} (M a_\ell)$. We show that there exists $f: N_1 \cong_M N_2$ with $f (a_1) = a_2$. We use axiom II(\ref{ii-1}), where $G, G', H, H', g, x, x'$ there stand for $M, M, N_1, N_2, \id_M, a_1, a_2$ here. We get that $\id_M \cup \{a_1, a_2\}$ is a partial embedding from $N_1$ to $N_2$. Now use axiom II(\ref{ii-2}) $\omega$-many times (as in the second paragraph of this proof) to extend this partial embedding to an isomorphism $f: N_1 \cong_M N_2$. By construction, we will have that $f(a_1) = a_2$, as desired.

  Finally, it is straightforward to see that (\ref{qm-5}) holds if and only if $\C$ is unbounded, as desired.
\end{proof}

We now examine the other direction: any quasiminimal AEC is isomorphic (as a concrete category) to a quasiminimal pregeometry class. Let us describe the proof. We start with a quasiminimal AEC $\K$. We first prove several semantic properties of this class. The class has amalgamation and joint embedding in $\aleph_0$. This is essentially because the uniqueness of the generic type together with the characterization of Galois types in Fact \ref{gtp-eq} allow us to amalgamate ``point by point''. By uniqueness of the generic type, $\K$ is also stable in $\aleph_0$. Now let us assume for simplicity that $\K$ has no maximal models in $\aleph_0$. Then stability implies that $\K$ has a saturated model $M$ of cardinality $\aleph_1$. By saturation, the closure operator $\cl^M$ inside $M$ is $\aleph_1$-homogeneous, and therefore we can apply the results of Section \ref{sec-2}. Using the finite character property of the closure operator (Fact \ref{cl-kappa}), We can conclude that $\cl^N$ is a pregeometry for any $N \in \K$. This takes care of axiom I(\ref{i-1}) in the definition of a quasiminimal pregeometry class. 

Now since the axioms of a quasiminimal pregeometry class deal with quantifier-free types, we add a relation for each Galois types (of finite length) over the empty set and expand $\K$ to a new AEC $\bigK$ where finite Galois types coincide with quantifier-free types. It is then easy to prove most of the axioms of a quasiminimal pregeometry class: only II(\ref{ii-2}), a form of $\aleph_0$-homogeneity, is problematic. It is easy to show that it holds when $H$ is empty, but at this stage we do not know whether \emph{countable} quantifier-free types coincide with Galois types. Using stability and amalgamation, we \emph{do} know that every Galois type over a countable model does not split over a finite set (in an appropriate sense). This is known to be enough to prove II(\ref{ii-2}) \cite[5.3]{quasimin-five}.

Let us implement the above description of the proof. First, quasiminimal AECs have (for countable models) amalgamation and joint embedding:

\begin{lem}\label{ap-lem}
  If $\K$ is a quasiminimal AEC, then $\K_{\le \aleph_0}$ has amalgamation and joint embedding.
\end{lem}
\begin{proof}
  We prove amalgamation, and joint embedding can then be obtained from the existence of the prime model and some renaming. By the ``in particular'' part of \cite[4.14]{ap-universal-apal}, it is enough to prove the so-called type extension property in $\K_{\le \aleph_0}$. This is given by the following claim: 

  \underline{Claim}: If $M \lea N$ are both in $\K_{\le \aleph_0}$ and $p \in \gS (M)$, then there exists $q \in \gS (N)$ extending $p$.

  \underline{Proof of Claim}: Say $p = \gtp (a / M; N')$. If $a \in |M|$ (i.e.\ $p$ is algebraic), let $q := \gtp (a / N; N)$. Assume now that $a \notin |M|$. If $M = N$, take $q = p$, so assume also that $M \lta N$. Let $b \in |N| \backslash |M|$ and let $p' := \gtp (b / M; N)$. By uniqueness of the generic type, $p' = p$. Therefore $q := \gtp (b / N; N)$ is as desired. $\dagger_{\text{Claim}}$
\end{proof}

We obtain the following equivalent definition of a quasiminimal AEC:

\begin{thm}
  Let $\K$ be an AEC satisfying (\ref{qm-1}), (\ref{qm-3}), and (\ref{qm-4}) from Definition \ref{quasimin-def}. Then (\ref{qm-2}) is equivalent to:

  \begin{itemize}
  \item[(\ref{qm-2})'] $\K \neq \emptyset$ and $\K_{\le \aleph_0}$ has joint embedding.
  \end{itemize}
\end{thm}
\begin{proof}
  That (\ref{qm-2}) implies (\ref{qm-2})' is given by Lemma \ref{ap-lem}. For the other direction, one can use joint embedding to see that $\cl^M (\emptyset)$ is a prime model for any $M \in \K_{\le \aleph_0}$.
\end{proof}

It directly follows that quasiminimal AECs are $\aleph_0$-stable:

\begin{lem}\label{stable-lem}
  If $\K$ is a quasiminimal AEC, then $\K$ is (Galois) stable in $\aleph_0$.
\end{lem}
\begin{proof}
  By uniqueness of the generic type.
\end{proof}

We can now show that the closure operator in a quasiminimal AEC satisfies exchange:

\begin{thm}\label{pregeom-thm}
  If $\K$ is a quasiminimal AEC, then $\K$ admits intersections and for any $N \in \K$, $(|N|, \cl^N)$ is a pregeometry whose closed sets are exactly the $\K$-substructures of $N$.
\end{thm}
\begin{proof}
  That $\K$ admits intersection is Theorem \ref{intersec-up}. Now let $N \in \K$ and let $W := (|N|, \cl^N)$. By Lemma \ref{intersec-closure-lem}, $W$ is a closure space and by Fact \ref{intersec-charact}, its closed sets are exactly the $\K$-substructures of $N$. By Fact \ref{cl-kappa}, $\kappa (W) = \aleph_0$, i.e.\ $W$ has finite character. It remains to see that $W$ has exchange. Let $a, b \in |N|$ and let $A \subseteq |N|$. Assume that $a \in \cl^N (Ab) \backslash \cl^N (A)$. We want to see that $b \in \cl^N (Aa)$. By finite character we can assume without loss of generality that $|A| \le \aleph_0$. Using the Löwenheim-Skolem-Tarski axiom, we may also assume that $N \in \K_{\le \aleph_0}$.

  Using stability, let $N' \in \K_{\le \aleph_1}$ be such that $N \lea N'$ and $N'$ is $\aleph_1$-saturated (this can be done even if there is a countable maximal model above $N$. In this case such a maximal model will be the desired $N'$). Then $W' := (|N'|, \cl^{N'})$ is a closure space with $\kappa (W') = \aleph_0$ which (using uniqueness of the generic type) is $\aleph_1$-homogeneous. Therefore by Corollary \ref{pregeom-cor}, $W'$ satisfies exchange. It follows immediately (see Fact \ref{cl-restr}) that $W$ also satisfies exchange.
\end{proof}

In particular, exchange is not necessary in the definition of a quasiminimal pregeometry class:

\begin{cor}\label{exchange-not-necessary}
  If $\C$ satisfies all the axioms of a quasiminimal pregeometry class except that in I(\ref{i-1}) $\cl_H$ may not have exchange, then $\C$ is a quasiminimal pregeometry class.
\end{cor}
\begin{proof}
  By Theorem \ref{qm-aec}, $\K (\C)$ is a quasiminimal AEC. By Theorem \ref{pregeom-thm}, $(|M|, \cl^M)$ is a pregeometry for every $M \in \K$. The result now follows from Lemma \ref{uq-cl}.
\end{proof}
\begin{remark}
  A referee pointed out that \emph{if} we assume in addition that any countable model is contained in a model $M \in \K$ such that $M = \cl^M (X)$ with $|X|$ infinite and $\cl^M (Y_1) \cap \cl^M (Y_2) = \cl^M (Y_1 \cap Y_2)$ for any $Y_1, Y_2 \subseteq X$ (this is Zilber's definition of an independent set in \cite[1.2]{zil05}), \emph{then} exchange can be proven as follows: by ``successive renaming of $X$'' (see the proof of \cite[Theorem 2]{zil05}), one can prove that $\C$ contains a model of cardinality $\aleph_2$ (and indeed of any cardinality). Then one can apply the results from \cite[2.8]{itai-tsuboi-wakai} (or Corollary \ref{pregeom-cor} here). Corollary \ref{exchange-not-necessary} is much stronger, as it does not even assume that $\C$ is unbounded, let alone that the infinite-dimensional model satisfies a weak version of exchange.
\end{remark}

In order to prove that axiom II(\ref{ii-2}) holds in an appropriate expansion of $\K$, we will use that the members of $\K$ are homogeneous for finite Galois types:

\begin{lem}\label{hom-lem}
  Let $\K$ be a quasiminimal AEC and let $M \in \K$. Let $\ba_1, \ba_2, \bb_1 \in \fct{<\omega}{|M|}$. If $\gtp (\ba_1 / \emptyset; M) = \gtp (\ba_2 / \emptyset; M)$, then there exists $\bb_2 \in \fct{<\omega}{|M|}$ such that $\gtp (\ba_1 \bb_1 / \emptyset; M) = \gtp (\ba_2 \bb_2 / \emptyset; M)$.
\end{lem}
\begin{proof}
  It is enough to prove the result when $\ell (\bb_1) = 1$. Let $M_\ell := \cl^{M} (\ba_\ell)$ for $\ell = 1,2$. By Fact \ref{gtp-eq}, there exists $f: M_1 \cong M_2$ such that $f (\ba_1) = \ba_2$. If $b_1 \in |M_1|$, let $b_2 := f (b_1)$ and check that this works. If $b_1 \in |M| \backslash |M_1|$, then $M_1 \neq M$, and so $M_2 \neq M$, so pick any $b_2 \in |M| \backslash |M_2|$. By uniqueness of the generic type, Fact \ref{gtp-eq}, and some renaming, there is an extension $g: \cl^M (M_1 b_1) \cong \cl^M (M_2 b_2)$ of $f$ sending $b_1$ to $b_2$, as desired.
\end{proof}

We will also use the following very general fact. The proof is similar to \cite[I.5.6]{shelahaecbook} or \cite[4.2]{quasimin-five}.

\begin{fact}\label{ns-fact}
  Let $\K$ be an AEC with $\LS (\K) = \aleph_0$ such that $\K_{\le \aleph_0}$ has amalgamation and $\K$ is stable in $\aleph_0$. Let $M \in \K_{\le \aleph_0}$. If $M$ satisfies the conclusion of Lemma \ref{hom-lem} (e.g.\ if $M$ is limit or if $\K$ is quasiminimal), then for any $p \in \gS^{<\omega} (M)$, there exists a finite $A \subseteq |M|$ such that $p$ does not split over $A$. That is, whenever $p = \gtp (\bb / M; N)$, if $\bb_1, \bb_2 \in \fct{<\omega}{|M|}$ are such that $\gtp (\bb_1 / A; N) = \gtp (\bb_2 / A; N)$, then $\gtp (\bb \bb_1  / A; N) = \gtp (\bb \bb_2 / A; N)$.
\end{fact}

We can now state and prove the correspondence between quasiminimal AECs and quasiminimal pregeometry classes. The idea is to add a relation to the language for each finite Galois type, and expand the models accordingly. This is a functorial process: the resulting class is isomorphic (as a category to the original one). This expansion is what we call the $(<\aleph_0)$-Galois-Morleyization \cite[3.3]{sv-infinitary-stability-afml}.

\begin{defin}\label{def-galois-m}
Let $\K = (K, \lea)$ be an AEC. Define an expansion $\bigtau$ of $\tau (K)$ by adding a relation symbol $R_p$ of arity $\ell (p)$ for each $p \in \gS^{<\omega} (\emptyset)$. Expand each $N \in K$ to a $\bigtau$-structure $\bigN$ by specifying that for each $\ba \in \fct{<\omega}{|\bigN|}$, $R_p^{\bigN} (\ba)$ (where $R_p^{\bigN}$ is the interpretation of $R_p$ inside $\bigN$) holds exactly when $\gtp (\ba / \emptyset; N) = p$. Let $\widehat{K}$ be the class of all such $\bigN$. For $\bigM, \bigN \in \widehat{K}$, write $\bigM \leap{\bigK} \bigN$ if $\bigM \subseteq \bigN$ and $\bigM \rest \tau (\K) \lea \bigN \rest \tau (\K)$. We call $\bigK := (\widehat{K}, \leap{\bigK})$ the \emph{$(<\aleph_0)$-Galois Morleyization} of $\K$.
\end{defin}

The basic facts about the Galois Morleyization that we will use are below. The most important says that finite Galois types are the same as quantifier-free types in the Galois Morleyization.

\begin{fact}\label{morleyization-facts}
  Let $\K$ be an AEC and let $\bigK = (\sbigK, \leap{\bigK})$ be its $(<\aleph_0)$-Galois Morleyization.

  \begin{enumerate}
  \item\cite[3.4]{sv-infinitary-stability-afml} $|\tau (\bigK)| = |\gS^{<\omega} (\emptyset)| + |\tau (\K) |$.
  \item\cite[3.5]{sv-infinitary-stability-afml} $\bigK$ is a functorial expansion of $\K$. This means that the reduct map is an isomorphism of concrete categories from $\bigK$ onto $\K$. In particular, $\bigK$ is an AEC with $\LS (\bigK) = \LS (\K) + |\tau (\bigK)|$.
  \item\label{morleyization-3} \cite[3.12, 3.16]{sv-infinitary-stability-afml} For any $N_1, N_2$, any $A \subseteq |N_1| \cap |N_2|$, and any $\bb_1 \in \fct{<\omega}{|N_1|}$, $\bb_2 \in \fct{<\omega}{|N_2|}$. If $\gtp (\bb_1 / A; N_1) = \gtp (\bb_2 / A; N_2)$, then the quantifier-free type of $\bb_1$ over $A$ in $N_1$ equals the quantifier-free type of $\bb_2$ over $A$ in $N_2$. If $A$ is finite, the converse also holds.
  \end{enumerate}
\end{fact}

We have arrived to the definition of the correspondence between quasiminimal AECs and quasiminimal pregeometry classes, and the proof that it works:

\begin{defin}
  For $\K$, a quasiminimal AEC let $\C (\K)$ be the class $\{(M, \cl^M) \mid M \in \bigK\}$, where $\bigK$ is the $(<\aleph_0)$-Galois Morleyization of $\K$.
\end{defin}

\begin{thm}\label{aec-qm}
  If $\K$ is a quasiminimal AEC, then $\C (\K)$ is a quasiminimal pregeometry class, which is unbounded if and only if $\K$ is. Moreover $\K (\C (\K))$ is the $(<\aleph_0)$-Galois Morleyization of $\K$.
\end{thm}
\begin{proof}
  Let $\C := \C (\K)$. It is clear that the elements of $\C$ are of the right form. The moreover part is clear from the definition of $\K (\C (\K))$, and so is the equivalence between the two versions of being bounded. We check all the conditions of Definition \ref{quasimin-class-def}. We will use without comments that  $\K_{\le \aleph_0}$ is has amalgamation and joint embedding (Lemma \ref{ap-lem}) and is stable in $\aleph_0$ (Lemma \ref{stable-lem}).

    \begin{itemize}
    \item[0: ]
    \begin{enumerate}
    \item Since $\K$ has a prime model (axiom (\ref{qm-2}) in Definition \ref{quasimin-def}), $\C \neq \emptyset$. By Fact \ref{morleyization-facts}, $|\tau (\C)| = |\tau (\bigK)| \le |\gS^{<\omega} (\emptyset)| + |\tau (\K)|$. Since $\LS (\K) = \aleph_0$, we have that $|\tau (\K)| \le \aleph_0$. Using that $\K \neq \emptyset$, pick $M \in \K_{\le \aleph_0}$. Since $\K_{\le \aleph_0}$ has amalgamation and joint embedding, there is an injection from $\gS^{<\omega} (\emptyset)$ into $\gS^{<\omega} (M)$. By amalgamation and stability, there exists $M' \in \K_{\aleph_0}$ universal over $M$. Therefore $|\gS^{<\omega} (M)| \le \aleph_0$. Thus $|\tau (\bigK)| \le \aleph_0$, as desired.
    \item This is clear. In fact, if $f: M \cong N$, then by definition of $\cl^M$ and $\cl^N$, $f$ is automatically an isomorphism from $(|M|, \cl^M)$ onto $(|N|, \cl^N)$.
    \item Let $(M, \cl^M), (N, \cl^N) \in \C$. If $M_0 \leap{\bigK} M$, then $M_0 \subseteq M$ so $M_0$ and $M$ satisfy the same quantifier-free sentences, so without loss of generality $M$ and $N$ are already countable. Now use that $\K_{\le \aleph_0}$ has joint embedding (by Lemma \ref{ap-lem}).
    \end{enumerate}
  \item[I: ]
    \begin{enumerate}
    \item Let $(M, \cl^M) \in \C$. By Theorem \ref{pregeom-thm}, $(M, \cl^M)$ is a pregeometry. Moreover if $A \subseteq |M|$ is finite then $|\cl^M (A)| \le \LS (\K) = \aleph_0$, as desired.
    \item Let $(M, \cl^M) \in \C$ and $X \subseteq |M|$. By definition of admitting intersections, $\cl^M (X) \lea M$ and so the result follows.
    \item Let  $(M, \cl^M), (M', \cl^{M'}) \in \C$, $X \subseteq |M|$, $y \in \cl^M (X)$, and $f: M \rightharpoonup M'$ be a partial embedding with $X \cup \{y\} \subseteq \preim (f)$. We want to see that $f (y) \in \cl^{M'} (f[X])$. By finite character, we may assume without loss of generality that $X$ is finite and therefore $\preim (f)$ is also finite. Let $\ba$ be an enumeration of $X$. Since quantifier-free types and Galois types over finite sets coincide in $\bigK$ (Fact \ref{morleyization-facts}(\ref{morleyization-3})), $\gtp (\ba y / \emptyset; M) = \gtp (f(\ba) f (y) / \emptyset; M')$. The result now follows from the definition of the closure operator.
    \end{enumerate}
  \item[II: ]
    Let $(H, \cl_{H}), (H', \cl_{H'}) \in \C$. Let $G \subseteq H$ and $G' \subseteq H'$ be countable closed subsets or empty and let $g: G \rightarrow G'$ be an isomorphism.
    \begin{enumerate}
    \item Let $x \in |H|$ and $x' \in |H'|$ be independent from $G$ and $G'$ respectively. We show that $g \cup \{(x, x')\}$ is a partial embedding. By renaming without loss of generality $G = G'$. By uniqueness of the generic type, $\gtp (x / G; H)  = \gtp (x' / G; H')$. The result follows, since Galois types are always finer than quantifier-free types (Fact \ref{morleyization-facts}(\ref{morleyization-3})).
    \item We use \cite[5.3]{quasimin-five}. It says that (assuming the axioms of quasiminimal pregeometry classes that we have established already) II(\ref{ii-2}) follows from the conclusion of \cite[2.2]{quasimin-five} and \cite[4.2]{quasimin-five}. Here, the first is proven as Lemma \ref{hom-lem} and the second as Fact \ref{ns-fact} (recalling that finite Galois types and quantifier-free types coincide, Fact \ref{morleyization-facts}(\ref{morleyization-3})).
    \end{enumerate}
  \item[IV: ] 
    \begin{enumerate}
    \item Because $\K$ is an AEC and the closure operator has finite character.
    \end{enumerate}
  \end{itemize}
\end{proof}

As a corollary of Theorem \ref{aec-qm}, all the work on structural properties of quasiminimal pregeometry classes automatically applies also to quasiminimal AECs:

\begin{cor}\label{final-cor}
  Let $\K$ be a quasiminimal AEC.
  \begin{enumerate}
  \item\label{final-cor-0} $\K$ is $(<\aleph_0)$-tame for types of finite length (that is, for any two distinct $p, q \in \gS^{<\omega} (M)$, there exists a finite $A \subseteq |M|$ such that $p \rest A \neq q \rest A$).
  \item\label{final-cor-1} Let $M, N \in \K$ and let $B_M, B_N$ be bases for $(|M|, \cl^M)$ and $(|N|, \cl^N)$ respectively. If $f$ is a bijection from $B_M$ onto $B_N$, then there exists an isomorphism $g: M \cong N$ with $f \subseteq g$.
  \item\label{final-cor-2} If $\K$ is unbounded, then:
    \begin{enumerate}
      \item $\K$ has no maximal models.
      \item $\K$ has exactly $\aleph_0$ non-isomorphic countable models and $\K$ is categorical in every uncountable cardinal.
    \end{enumerate}
  \end{enumerate}
\end{cor}
\begin{proof}
  Let $\C := \C (\K)$. By Theorem \ref{aec-qm}, $\C$ is a quasiminimal pregeometry class. By \cite{quasimin-five}, it also satisfies the excellence axiom. By Zilber's main result on these classes \cite{zil05} (or see \cite{quasimin} for an exposition), (\ref{final-cor-0}) and (\ref{final-cor-1}) hold for $\C$. Therefore they also hold for $\K (\C)$, which is a just a functorial expansion of $\K$. Hence they also hold for $\K$. Similarly, (\ref{final-cor-2}) holds in unbounded quasiminimal AECs (see \cite[\S4]{quasimin}).
\end{proof}

\bibliographystyle{amsalpha}
\bibliography{quasimin}

\end{document}